\documentclass[12pt]{amsart}

\usepackage{
amssymb,
amsthm,
amsmath
}

\usepackage[pdftex]{graphicx}
\DeclareGraphicsExtensions{.pdf, .jpg, .tif}

\renewcommand{\epsilon}{\varepsilon}

\usepackage{color}

\theoremstyle{plain}

\newtheorem{theorem}{Theorem}

\newtheorem{lemma}[theorem]{Lemma}

\newtheorem*{theorem*}{Theorem}
\newtheorem*{prop*}{Proposition}
\newtheorem*{corollary*}{Corollary}
\newtheorem*{lemma*}{Lemma}

\theoremstyle{definition}

\theoremstyle{remark}

\newtheorem{ex}[theorem]{Example}
\newtheorem{remark}[theorem]{Remark}

\numberwithin{equation}{section}
\numberwithin{theorem}{section}
\newcommand{\plus}{(+)}
\newcommand{\minus}{(-)}

\newcommand{\erho}{\rho}
\newcommand{\ephi}{\varphi}

\newcommand{\half}{\tfrac{1}{2}}

\usepackage[titletoc]{appendix}

\begin{document}

\title{Ergoregions between two ergospheres}

\author{Gregory Eskin}
\address{Department of Mathematics, UCLA, Los Angeles, CA 90095-1555}
\email{eskin@math.ucla.edu}

\author{Michael A Hall}
\address{Department of Mathematics, USC, Los Angeles, CA 90089-2532}
\email{mikehall1000@gmail.com}

\maketitle

\begin{center}
In blessed memory of Miroslav L'vovich Gorbachuk
\end{center}

\begin{abstract}
For a stationary spacetime metric, black holes are spatial regions which disturbances may not propagate out of. In our previous work an existence and regularity theorem was proven for black holes in two space dimensions, in the case where the boundary of the ergoregion is a simple closed curve surrounding a singularity. In this paper we study the case of an annular ergoregion, whose boundary has two components.
\end{abstract}

\section{Introduction}

On $\mathbb{R}^{1+2} \cong \mathbb{R}^1_t \times \mathbb{R}^2_x$, let $g$ be a stationary pseudo-Riemannian (Lorentzian) metric with signature $(+1,-1,-1)$, and consider the associated wave equation $\Box_g u = 0$. 

As we take the general point of view of \emph{analogue spacetimes}, we largely ignore issues of coordinate invariance, working in the global system of coordinates $(t,x)$. For convenience we sometimes write $t = x^0$, $x = (x^1,x^2)$, and denote the corresponding components of the metric $g_{ij} = g_{ij}(x)$, $0 \leq i,j \leq 2$. Here, the assumption that $g$ is stationary means these depend only on $x$. 
Inverting the matrix of the $g_{ij}$ gives the components of the inverse metric, $g^{ij} = g^{ij}(x)$. We assume $g^{00}(x) > 0$ for all $x$, which corresponds to the natural time orientation.

The wave equation is then 
\begin{equation}\label{wave}
\Box_g u = \sum_{i,j=0}^2 \frac{1}{\sqrt{|g(x)|}} \partial_{x^i} [\sqrt{|g(x)|} g^{ij}(x) \partial_{x^j} u] = 0,
\end{equation}
where $|g(x)| = |\det[ g_{ij} ]_{i,j=0}^2 |$.

Denote by $\Omega$ the \emph{ergoregion}, which is the set of $x \in \mathbb{R}^2$ where $g_{00}(x) < 0$, i.e. $\partial_t$ is not timelike. In other words, the ergoregion is the region where the spatial part of the wave operator is not elliptic. Note that by Cramer's rule and our sign conventions we have $g_{00}(x) = g^{00}(x)\Delta(x)$, where $\Delta = g^{11}g^{22} - (g^{12})^2$ and $g^{00}(x) > 0$, so $\Omega = \{ \Delta(x) < 0 \}$. 

In [4] we discussed the case of an ergosphere surrounding a singularity. In this paper we consider the case of an annular domain $\Omega$, i.e. one whose boundary $\partial \Omega = \{ \Delta(x) = 0 \}$ consists of two nested Jordan curves, each of which is smooth in the sense that $\partial \Delta/\partial x \neq 0$ when $\Delta = 0$. Informally we say that there are ``two ergospheres''. Assume the components of the metric $g$ are defined in a fixed spatial neighborhood of $\overline{\Omega}$.

\subsection{Examples of spacetimes with two ergospheres}

An example of a metric in three space dimensions having two ergospheres is the celebrated Kerr metric, which we will write in Kerr-Schild coordinates [7],
[12],  [13].

The Hamiltonian in cylindrical coordinates $(\erho,\ephi,z)$ is given by 
\begin{align}
H = \tau^2 - \xi_\erho^2 - (\tfrac{1}{\erho}\xi_\ephi)^2 - \xi_z^2 + K(-\tau + \hat b \cdot \hat \xi)^2, 
\end{align}
where $\hat \xi = (\xi_\erho, \tfrac{1}{\erho}\xi_\ephi,\xi_z)$, $\hat b = (b_\erho,b_\ephi,b_z)$, with
\begin{align}
K = K(\erho,z) = \frac{2mr^3}{r^4 + a^2z^2}, \quad b_\erho = \frac{\erho r}{r^2 + a^2}, \quad b_\ephi = \frac{a \erho}{r^2 + a^2}, \quad b_z = \frac{z}{r}.
\end{align}
Here $r$ is defined by the relation
\begin{align}
\frac{\erho^2}{r^2 + a^2} + \frac{z^2}{r^2} = 1.
\end{align}
Therefore note that $b_\erho^2 + b_\ephi^2 + b_z^2 = 1$. 

For the Kerr metric it is well known that there are two ergospheres called the \textit{outer ergosphere} and \textit{inner ergosphere}, which occur where $K = 1$. There are also outer and inner horizons given by the equations $r = r_\pm = m \pm \sqrt{m^2 - a^2}$, assuming $0 < a < m$. 

Two examples of 2+1 spacetimes with two ergospheres may be obtained by reduction from the 3+1 dimensional Kerr metric:
\begin{itemize}
\item Set $z = 0$ and $\xi_z = 0$ to obtain the Hamiltonian
\begin{align}
H_1= \tau^2 - \xi_\erho^2 - (\tfrac{1}{\erho}\xi_\ephi)^2 + K(-\tau + b_\rho\xi_\rho + b_\varphi (\tfrac{1}{\erho}\xi_\ephi))^2, 
\end{align}
where $K = \frac{2m}{r} = \frac{2m}{\sqrt{\erho^2 - a^2}}$. This corresponds to the equatorial plane of the Kerr metric.

\item Since the Kerr Hamiltonian $H$ is independent of $\ephi$, we have that $\xi_\ephi$ is constant. If we set $\xi_\ephi = 0$, then we obtain the Hamiltonian
\begin{align}\label{eqn1.5}
H_2 = \tau^2 - \xi_\erho^2 - \xi_z^2 + K(-\tau + b_\erho\xi_\erho  + b_z\xi_z)^2
\end{align}
This corresponds to a reduction by rotational symmetry. 
\end{itemize}

Related to the first situation we consider so-called
 \textit{acoustic metrics} with Hamiltonians
\begin{align}
H = \left( \tau + A\xi_\erho + B(\tfrac{1}{\erho}\xi_\ephi) \right)^2 - \xi_\erho^2 - (\tfrac{1}{\erho}\xi_\ephi)^2
\end{align}
with $A = A(\erho)$, $B = B(\erho)$  (cf [11], [4]).
Here the ergosphere is where $A^2 + B^2 = 1$, and we consider the case where the ergoregion in an annular domain between $\erho = \erho_-$ and $\erho = \erho_+$.

Another example is the optical (Gordon) metric [8], [5], whose inverse metric tensor has components
\[
g^{ij} = \eta^{ij} + (n(x)^2 - 1)u^iu^j, 
\]
where $\eta^{ij}$ is the inverse of the Minkowski metric tensor with signature $(+1,-1,-1)$; $n(x)$ is the index of refraction, which describes the propagation of light in a moving dielectric; $(u^0,u^1,u^2,u^3) = (1-	\frac{|w|^2}{c^2})(1,\frac{w}{c})$ is the 4-velocity of the medium flow, with $w(x) = (w_1,w_2,w_3)$ the velocity of the dielectric; and $c$ is the speed of light in a vacuum.

The Gordon Hamiltonian is given by 
\begin{align}
\textstyle H = \tau^2 - \sum_{j=1}^3 \xi_j^2 + (n(x)^2 - 1)\left( \sum_{j=1}^3 u^j \xi_j \right)^2.
\end{align}
The plan of the paper is as follows. 
In Section 2 we review the role of so-called `zero-energy' null geodesics.
In Section 3 we consider the first example above of a reduction of the Kerr metric. 
In Section 4 we generalize the results of Section 3 to the case of acoustic metrics. 

In Section 5, we will consider acoustic metrics where we allow double roots and more than two horizons.
In Section 6, we give some results for a general metric on a 2+1 dimensional spacetime with two ergospheres, describing some possible behaviors 
of the zero energy null geodesics. In Section 7 we consider the case where both ergospheres are also horizons. 

\section{Zero energy null geodesics in the ergoregion}

On the cotangent space $T^* \mathbb{R}^{1+2}$ we use global coordinates $(t,x,\tau,\xi)$, often denoting $t = x^0$, $x = (x^1,x^2)$, $\tau = \xi_0$, 
and $\xi = (\xi_1,\xi_2)$.

As in the analysis of [2],  [4]  
the main idea will be to analyze the dynamics of the \emph{zero energy null geodesics}. Recall that for a point $x \in \Omega$, the forward light-cone at $x$ consists of null-bicharacteristics with increasing $t$. Its spatial projection is a cone based at $x \in \mathbb{R}^2$, whose edges correspond to null-bicharacteristics with $\tau = 0$ (thus `zero energy'). These edges may be described by pair of smooth, autonomous vector fields $X^\pm = X^\pm(x)$, $x \in \Omega$, which give the corresponding null-geodesic flow parameterized by $t$. 

Explicitly, let $\sigma(x,\xi) = \sum_{i,j=0}^2 g^{ij}(x) \xi_i\xi_j$ be the symbol 
of $\Box_g$, and consider a general bicharacteristic curve $(x,\xi) = (x(t),\xi(t)) \in T^* \mathbb{R}^{1+2}$, parameterized by $t$, i.e. 
\begin{align}\label{2.1}
\begin{split}
\frac{dx^i}{dt}
&= \frac{\partial_{\xi_i} \sigma(x,\xi) }{\partial_{\tau} \sigma(x,\xi)} 
= \frac{\sum_{j=0}^2 2g^{ij}(x)\xi_j}{\sum_{j=0}^2 g^{0j}(x) \xi_j}, \\
\frac{d\xi_i}{dt}
&= \frac{\partial_{x^i} \sigma(x,\xi)}{\partial_{\tau} \sigma(x,\xi)}
= \frac{\sum_{i,j=0}^2 \partial_{x^i} g^{ij}(x)\xi_i\xi_j}{\sum_{j=0}^2 g^{0j}(x) \xi_j}.
\end{split}
\end{align}
Note that as the metric is stationary, $\xi_0 = \tau$ is constant. Then for a null-bicharacteristic with $\tau = 0$ we have the characteristic equation 
\begin{align}
\sum_{i,j=1}^2 g^{ij}(x) \xi_i\xi_j = 0
\end{align}
We may solve for $\xi_1$ to obtain $\xi_1 = \frac{-g^{12} \pm \sqrt{-\Delta}}{g^{11}}\xi_2$, and substituting this relation and $\tau = 0$ into the above, we obtain 
\begin{align}\label{Xis}
\frac{dx^i}{dt} 
= 2 \frac{g^{i1} \frac{-g^{12} \pm \sqrt{-\Delta} }{g^{11}} + g^{i2}}{g^{01}\frac{-g^{12} \pm \sqrt{-\Delta} }{g^{11}} + g^{02}} 
= 2 \frac{g^{11}g^{i2}  - g^{12}g^{i1} \pm  g^{i1}\sqrt{-\Delta}}{g^{11}g^{02} - g^{12}g^{01} \pm   g^{01}\sqrt{-\Delta}} =: X^{\pm,i}(x).
\end{align}
Note that the choice of sign is arbitrary, but we may make a consistent choice throughout the ergoregion. 
To analyze the examples introduced in Section 1, we will study the dynamics of the vector fields $X^\pm$.

We recall from the analysis of [4] the following fact:
\begin{lemma}\label{horizonlemma}
If $\gamma$ is a limit cycle for one of the vector fields $X^\pm$, then $\gamma$ is a horizon. 
\end{lemma}

\section{The case of the equatorial plane for the Kerr metric}

We consider our first example of a 2+1 spacetime obtained by reduction from the Kerr metric.
Recall that we set $z = 0$ and $\xi_z = 0$ to obtain the Hamiltonian
\begin{align}\label{2.1}
H = \tau^2 - \xi_\erho^2 - (\tfrac{1}{\erho}\xi_\ephi)^2 + K(-\tau + b_\erho\xi_\erho + b_\ephi (\tfrac{1}{\erho}\xi_\ephi))^2, 
\end{align}
where $K = \frac{2m}{r} = \frac{2m}{\sqrt{\erho^2 - a^2}}$. We consider the region $a < \erho < \sqrt{4m^2 + a^2}$. 

We calculate
\begin{align*}
\half H_{\tau} &= \tau - K(-\tau + b_\erho \xi_\erho + b_\ephi  (\tfrac{1}{\erho}\xi_\ephi))\\
\half H_{\xi_\erho} &= -\xi_\erho + K(-\tau+b_\erho \xi_\erho + b_\ephi  (\tfrac{1}{\erho}\xi_\ephi))b_\erho \\
\half H_{\xi_\ephi} &= -\tfrac{1}{\erho}(\tfrac{1}{\erho}\xi_\ephi) + K(-\tau + b_\erho\xi_\erho +  b_\ephi(\tfrac{1}{\erho}\xi_\ephi)) (\tfrac{1}{\erho}b_\ephi),
\end{align*}
and therefore the zero energy null geodesic flow is given by 
\begin{align}\label{2.2}
\begin{split}
\frac{d\erho}{dt} 
&= \frac{-\xi_\erho + K(b_\erho \xi_\erho + b_\ephi (\tfrac{1}{\erho}\xi_\ephi))b_\erho}{-K(b_\erho \xi_\erho + b_\ephi (\tfrac{1}{\erho}\xi_\ephi))}\\
\frac{d\ephi}{dt} 
&= \frac{-\tfrac{1}{\erho}(\tfrac{1}{\erho}\xi_\ephi) + K(b_\erho\xi_\erho + b_\ephi (\tfrac{1}{\erho}\xi_\ephi))(\tfrac{1}{\erho}b_\ephi)}{-K(b_\erho \xi_\erho + b_\ephi (\tfrac{1}{\erho}\xi_\ephi))}
\end{split}
\end{align} 
The characteristic equation is
\begin{align}
(Kb_\erho^2 - 1) \xi_\erho^2 + 2K b_\erho b_\ephi \xi_\erho(\tfrac{1}{\erho}\xi_\ephi) + K(b_\ephi^2 - 1) (\tfrac{1}{\erho}\xi_\ephi)^2 = 0.
\end{align}
Solving for $\xi_\erho$ in terms of $\xi_\ephi$ yields
\begin{align}\label{2.4}
\begin{split}
\xi_\erho 
&= \frac{- Kb_\erho b_\ephi \pm \sqrt{K^2b_\erho^2 b_\ephi^2 - (Kb_\erho^2 - 1)(Kb_\ephi^2 - 1)}}{Kb_\erho^2 - 1} (\tfrac{1}{\erho}\xi_\ephi) \\
&= \frac{- Kb_\erho b_\ephi \pm \sqrt{K(b_\erho^2 + b_\ephi^2) - 1}}{Kb_\erho^2 - 1} (\tfrac{1}{\erho}\xi_\ephi)
\end{split}
\end{align}

The equation for zero energy null bicharacteristics becomes
\begin{align}\label{eqn2.5}
\begin{split}
\frac{d\erho^\pm}{dt} &= \frac{\pm(Kb_\erho^2-1)\sqrt{K(b_\erho^2 + b_\ephi^2) -1}}{\mp K b_\rho\sqrt{K(b_\erho^2 + b_\ephi^2) - 1} - K b_\ephi}.
\end{split}
\end{align}

The $\plus$ family are the solutions for the positive square root, and similarly for the $\minus$ family.

\begin{theorem}
The $\plus$ family produces two horizons, while
and the $\minus$ family does not produce any.
\end{theorem}

\begin{proof}
With the $\plus$ family, the denominator of (3.5) does not vanish, while  the numerator is zero when $Kb_\erho^2 - 1 = 0$. Since we have $K = \frac{2m}{r}$, $b_\erho = \frac{r}{\erho}$, this happens when $r = m \pm \sqrt{m^2 - a^2}$. These zeros are limiting cycles, and the limiting cycle is an horizon in each case. 

With the $\minus$ family, since the denominator vanishes we reorganize (3.5) to get
\begin{align}
\frac{d\erho^-}{dt} = \frac{-\sqrt{K(b_\erho^2 + b_\ephi^2) - 1}(Kb_\erho \sqrt{K(b_\erho^2 + b_\ephi^2) - 1} + K b_\ephi)}{K^2(b_\erho^2 + b_\ephi^2)}
\end{align}
which is negative everywhere in the ergoregion. Thus $\erho^-(t)$ is strictly decreasing, and there is no horizon generated by the $\minus$ family. 

\end{proof}

\section{Acoustic metrics I}

We assume a metric of the form 
\begin{align}
H = (\tau + A\xi_\erho + B(\tfrac{1}{\erho}\xi_\ephi))^2 - \xi_\erho^2 - (\tfrac{1}{\erho}\xi_\ephi)^2
\end{align}
with $A = A(\erho)$, $B = B(\erho)$. The ergospheres are where $A^2 + B^2 = 1$. We assume that this happens at $\erho = \erho_-$ and $\erho = \erho_+$, and that $A^2 + B^2 > 1$ on the interval $\erho \in (\erho_-,\erho_+)$.

We calculate 
\begin{align}
\begin{split}
\half H_{\tau} &= (\tau + A\xi_\erho + B(\tfrac{1}{\erho}\xi_\ephi)) \\
\half H_{\xi_\erho} &= (\tau + A\xi_\erho + B(\tfrac{1}{\erho}\xi_\ephi))A - \xi_\erho\\
\half H_{\xi_\ephi} &= (\tau + A\xi_\erho + B(\tfrac{1}{\erho}\xi_\ephi))\tfrac{1}{\erho}B - \tfrac{1}{\erho}(\tfrac{1}{\erho}\xi_\ephi).
\end{split}
\end{align}

Thus  for $\tau=0$
\begin{align}
\begin{split}
\frac{d\rho}{d t}=\frac{(A\xi_\rho +B\frac{\xi_\varphi}{\rho})A-\xi_\rho}{A\xi_\rho+B\frac{\xi_\varphi}{\rho}}\\
\frac{d\varphi}{d t}=\frac{(A\xi_\rho +B\frac{\xi_\varphi}{\rho})\frac{B}{\rho}-\frac{\xi_\varphi}{\rho^2}}{A\xi_\rho+B\frac{\xi_\varphi}{\rho}}.
\end{split}
\end{align}

We get the characteristic equation
\begin{align}\nonumber
(A^2 - 1)\xi_\erho^2 + 2AB\xi_\erho (\tfrac{1}{\erho}\xi_\ephi) + (B^2 - 1) (\tfrac{1}{\erho}\xi_\ephi)^2 = 0,
\end{align}
which leads to
\begin{align}
\xi_\erho = \frac{-AB \pm \sqrt{A^2 + B^2 - 1}}{A^2 - 1} (\tfrac {1}{\erho}\xi_\ephi).
\end{align}

We assume there are two values $\erho = \erho_i$, $i = 1,2$, such that $\erho_- < \erho_1 < \erho_2 < \erho_+$ and $A(\erho_i) = -1$, and we then claim that these are the horizons.

Let start with that $A < 0$ and $B > 0$. 
Substituting  (4.4)  into (4.3),  we get
\begin{align}\label{eqn3.5}
\begin{split}
\frac{d\erho^+}{dt} &= \frac{(A^2 -1) \sqrt{A^2 + B^2 - 1}}{A\sqrt{A^2 + B^2 -1} - B} \\
\frac{d\ephi^+}{dt} &= \frac{1}{\erho} \cdot \frac{\sqrt{A^2 + B^2 - 1}(AB - \sqrt{A^2 + B^2 - 1})}{A\sqrt{A^2 + B^2 - 1} - B}
\end{split}
\end{align}
and
\begin{align}\label{eqn3.6}
\begin{split}
\frac{d\erho^-}{dt} &= \frac{A(A^2 + B^2 - 1) - B\sqrt{A^2 + B^2 - 1}}{A^2 + B^2} \\
\frac{d\ephi^-}{dt} &= \frac{1}{\erho} \cdot\frac{\sqrt{A^2+B^2-1}(-AB-\sqrt{A^2 + B^2 -1})}
{-A\sqrt{A^2+B^2-1}-B}
\\
 &= \frac{1}{\erho} \cdot\frac{(A\sqrt{A^2 + B^2 -1}-B)(B^2-1)}{(-AB+\sqrt{A^2 + B^2 -1})(A^2 + B^2)}
\end{split}
\end{align}

To derive the above we use that 
\begin{align}
\begin{split}
&(A\sqrt{A^2 + B^2 - 1} - B)(-A\sqrt{A^2 + B^2 - 1} - B) = -(A^2 + B^2)(A^2 - 1),
\\
 &(-AB+\sqrt{A^2 + B^2 -1})(-AB-\sqrt{A^2 + B^2 -1})=(A^2-1)(B^2-1).
\end{split}
\end{align}

\begin{theorem}
Suppose $A < 0$ and $B > 0$. Then the $\plus$ family generates two horizons at $\erho = \erho_1$ and $\erho = \erho_2$, while the $\minus$ family generates no horizons.
\end{theorem}

The proof is the same as the proof of Theorem 2.1, with the horizons $\erho = \erho_1$, $\erho = \erho_2$ generated by the $\plus$ family when 
$A(\erho_1) = A(\erho_2) = -1$, and no horizons generated by the $\minus$ family. 

Consider now the case when $A < 0$ and $B < 0$. Then analogous to 
(4.6),  (4.7)   we obtain
\begin{align}
\begin{split}
\frac{d\erho^+}{dt} &= \frac{A(A^2 + B^2 - 1) + B\sqrt{A^2 + B^2 - 1}}{A^2 + B^2} \\
\frac{d\ephi^+}{dt} &= \frac{1}{\erho} \cdot\frac{(A\sqrt{A^2 + B^2 - 1}+B)(B^2-1)}{(-AB - \sqrt{A^2 + B^2 - 1})(A^2 + B^2)}
\end{split}
\end{align}
and
\begin{align}
\begin{split}
\frac{d\erho^-}{dt} &= \frac{-(A^2-1)\sqrt{A^2 + B^2 - 1}}{-A\sqrt{A^2 + B^2 - 1} - B} \\
\frac{d\ephi^-}{dt} &= \frac{1}{\erho} \cdot \frac{-AB\sqrt{A^2 + B^2 - 1} - (A^2 + B^2 - 1)}{-A\sqrt{A^2 + B^2 - 1} - B}
\end{split}
\end{align}

Thus we have:

\begin{theorem}
Suppose $A < 0$ and $B < 0$. Then the $\plus$ family generates no horizons 
since  $\frac{d\rho^+}{dt}<0$  on  $(\rho_-,\rho_+)$, while the $\minus$ family generates two horizons at $\erho = \erho_i$, $i = 1,3$ where $A(\erho_1) = A(\erho_2) = -1$.
\end{theorem}

In the case when $A < 0$ and $B$ changes sign between the two roots of $A = -1$, we have:

\begin{theorem}
Suppose $A(\erho_1) = A(\erho_2) = -1$ and $\erho_0 \in (\erho_1,\erho_2)$ is such that $B(\erho_0) = 0$, with $B(\erho) > 0$ when $\erho < \erho_0$ and $B(\erho) < 0$ when $\erho > \erho_0$.
Then the $\plus$ family produces one horizon, and the $\minus$ family also produces one horizon. 
\end{theorem}

Proof: Split into the intervals $(\erho_-,\erho_0)$ and $(\erho_0,\erho_+)$. Each subinterval contains one root of $A = -1$. By the analysis in Theorem 4.1 the $\plus$ family generates an horizon in the first subinterval while the $\minus$ family does not, and by the analysis in Theorem 4.2 the $\minus$ family generates an horizon in the second subinterval while the $\plus$ family does not. This proves Theorem 4.3.

\section{Acoustic metrics II}

In this section we wish to consider cases when $A = -1$ has more than two roots, or when it has roots of higher multiplicity. 

\begin{ex}
The case when $a = m$ is called the extremal Kerr metric. In this case the outer and inner horizons coincide in a single horizon $r= m$.
\end{ex}
 
More generally consider the case when $A = - 1$ has a double root $A(\erho_1) = -1$, $A'(\erho_1) = 0$, with say $A''(\erho_1) > 0$. 
$A(\erho) +1=C(\rho)(\rho-\rho_1)^2,\ C(\erho)>0$.   Suppose  $B>0$.

We have that for equation (4.5), $\erho^+(t) \equiv \erho_1$ is a solution and therefore a horizon. Also $\frac{d\erho^+}{dt} > 0$ when
 $\erho \neq \erho_1$, so $\erho^+(t)$ will spiral toward $\erho_1$ from below as $t \to +\infty$, or from above when $t \to -\infty$. Note that when $\erho$ approaches $\erho_1$, the trajectory spirals as a power instead of logarithmically as in the case of a simple root.

Now consider the case when $A = -1$ has more than two simple roots. There must be an even number $2m$ roots $\erho_{1}, \erho_{2}$, \ldots, $\erho_{2m-1},\erho_{2m}$. 

Take $\erho_{j},\erho_{j+1}$. If $B > 0$ on $[\erho_{j},\erho_{j+1}]$ then this is the same as the case of Theorem 4.1, so $\erho = \erho_j$, $\erho = \erho_{j+1}$ are horizons coming from the $\plus$ family, while there are no horizons from the $\minus$ family in this interval. If $B < 0$ then as in Theorem 4.2 $\erho = \erho_j$, $\erho = \erho_{j+1}$ are horizons from the $\minus$ family while there are none for the $\plus$ family. If $B$ changes sign then one of $\erho = \erho_j$, $\erho = \erho_{j+1}$ is a horizon from the $\plus$ family, and the other a horizon from the $\minus$ family as in Theorem 4.3.

If $A$ is always positive, we have the same picture when $A(\rho) = 1$, and the horizons are white.

\begin{remark} 
Note that our definition of a horizon is local. We say that a horizon is a black hole horizon if null geodesics cannot move from inside to outside, and a white hole horizon if they cannot move from outside to inside. 
\end{remark}

\section{The general case}

In this section we obtain several results in the general case.  As in section 2 we consider a general metric $g$ and we assume there are two ergospheres $\Gamma_1$ and $\Gamma_2$. 

We assume the boundary is never characteristic, so that the $\plus$ and $\minus$ families are transversal to the boundary. It follows from transversality that as $t$ increases one family starts on each $\Gamma_i$ and the other ends, $i = 1,2$.

We have all permutations of $\plus$ and $\minus$ starting and ending on $\Gamma_1$ and $\Gamma_2$. For concreteness assume $\plus$ starts on $\Gamma_1$ and $\minus$ ends there. For $\Gamma_2$, we can have that $\plus$ starts and $\minus$ ends, or the opposite. 

Let $\plus$ starts on both $\Gamma_1$ and $\Gamma_2$. Then 
any $\plus$ tracjectory starting on $\Gamma_1$  cannot end on $\Gamma_2$.  Thus
by the Poincar\' e-Bendixson theorem  (cf.  [6]), it approaches some   limit cycle $\gamma_1^+$, 
which must be a horizon by Lemma 2.1. 
Also   any  $\plus$  trajectory  starting  on $\Gamma_2$  approaching  some  limit  cycles  $\gamma_2^+$.   Note
$\gamma_1^+$  and  $\gamma_2^+$  may  coincide.   Similarly  if $\minus$ family ends both on  $\Gamma_1$  and   $\Gamma_2$   there is  at  least 
one limiting  cycle  belonging   to $\minus$  family.  Thus we get
\begin{theorem}
Suppose one of the families starts on both $\Gamma_1$ and $\Gamma_2$. Then to each of the $\plus$ and $\minus$ families corresponds at least one event horizon. 
\end{theorem}

%It may happen that there is no additional horizon produced by either family.

In the case where $\plus$ starts on $\Gamma_1$ and ends on $\Gamma_2$ there may be no event horizons produced by either family. The situation could be as in Section 4. 

The total contribution is the sum of the contributions of the $\plus$ and $\minus$ family. 

Thus, we have

\begin{theorem}
If the $\plus$ family starts on one ergosphere and ends on the other, then it may be that  the $\plus$ family  does not generate any event horizon. 
In this case  $\minus$   family also starts  on one ergosphere  and ends  on another.  Thus it  can be  no event  horizon produced by $\minus$  family.
\end{theorem}

Now we shall  compare  Theorems 6.1 and 6.2 with the results of \S 4. 
We start first with Theorem 4.3 where $A < 0$ and $B(\erho)$ changes sign: Say $B(\erho_0) = 0$, with $B(\erho) > 0$ when $\erho_- < \erho < \erho_0$, $B(\erho) <0$ when $\erho_0 < \erho < \erho_+$. Consider the $\plus$ family, for which we have (cf.  (4.5))
\begin{align}\label{6.1}
\frac{d\erho^+}{dt} \approx \frac{-B^2\sqrt{A^2+B^2-1}}{A\sqrt{A^2+B^2-1} -B} \approx B\sqrt{A^2 + B^2 -1} > 0
\end{align}
near $\erho = \erho_-$ since $A^2 + B^2 - 1 \approx 0,\ B>0$. 
Therefore $\erho_+(t)$ starts at $\erho = \erho_-$. Also $\frac{d\erho^+}{dt} < 0$ near $\erho = \erho_+$ since  $B<0$  near  $\rho=\rho_+$.

Thus $\erho_+(t)$ starts on both ergospheres at $\erho = \erho_-$, $\erho = \erho_+$. Hence the conditions of Theorem 6.1 are satisfied and so there are event horizons corresponding to both the $\plus$ and $\minus$ families. This was directly argued in Theorem 4.3. 

Consider now the case when in (4.5) we have $B > 0$ on $[\erho_-,\erho_+]$. Then it follows from (6.1)
 that $\frac{d\erho^+}{dt} > 0$ near $\erho = \erho_-$ and $\frac{d\erho^+}{dt} > 0$ near $\erho = \erho_+$. Therefore the $\plus$  family starts 
  at  $\rho=\rho_-$  and  ends at  $\rho=\rho_+$.  Also
$\minus$ family  ends at  $\rho=\rho_-$  and starts  at $\rho=\rho_+$.  Therefore   the conditions  of Theorem 6.2  are satisfied.
The results of Theorem 6.2  are in agreement  with the result of Theorem 4.1.

Now we shall  describe,  for example,  the phase portrait  of $\plus$  family  in $\Omega$.

Consider first the case of Theorem 4.3 or more generally of Theorem 6.1 with one $\plus$ horizon only.  Thus we have the event horizon $\erho = \erho_1$ 
where $\erho_- < \erho_1 < \erho_0$  and $A(\erho_1) = -1$, $B(\erho_1) > 0$. It follows from (4.5) at $\erho = \erho_1$ that 
\begin{align}\label{6.3}
\frac{d\ephi^+}{dt} = \frac{1}{\erho^+} \frac{B(\erho_1)(-B(\erho_1) - B(\erho_1))}{-B(\erho_1)-B(\erho_1)} = \frac{1}{\erho_1}B(\erho_1). 
\end{align}
Thus $\left.\frac{d\ephi^+}{dt}\right|_{\erho = \erho_1} > 0$. This means that the periodic $\plus$ trajectory $\erho = \erho_1$ is traversed counterclockwise as $t$ increases. Any $\plus$ trajectory $\gamma_1^+$ starting at $\erho = \erho_-$ approaches the horizon $\erho = \erho_1$ spiraling counterclockwise. Analogously any $\plus$ trajectory $\gamma_2^+$ starting at $\erho = \erho_+$ will also approach the horizon $\erho = \erho_1$ spiraling counterclockwise. 

Consider now the more involved case of $\plus$ trajectories in the conditions of Theorem 4.1. We have two event horizons $\erho_- < \erho_1 < \erho_2 < \erho_+$. As in (4.1) we have $\frac{d\ephi^+}{dt} > 0$ when $\erho = \erho_1$, i.e. the periodic trajectory $\erho = \erho_1$ is traversed counterclockwise as $t$ increases. At $\erho = \erho_2$ we have also $\frac{d\ephi^+}{dt} = \frac{1}{\erho_2}B(\erho_2) > 0$ since $B(\erho) > 0$ for all $\erho_- < \erho < \erho_+$. Therefore the horizon $\erho = \erho_2$ is also traversed counterclockwise.   Let $\gamma_0^+$ be any $\plus$ trajectory starting at $\erho = \erho_-$. 
It will approach $\erho = \erho_1$ spiraling counterclockwise as $t \to +\infty$. Consider any $\plus$ trajectory $\gamma_2^+$ ending at $\erho = \erho_+$. 
It will approach $\erho = \erho_2$ as $t \to -\infty$, spiraling clockwise around the horizon $\erho = \erho_2$. Note that $\erho = \erho_2$ is traversed counterclockwise as $t\to +\infty$ and the direction is reversed when $t\to -\infty$. Consider any $\plus$ trajectory between $\erho = \erho_1$ and $\erho = \erho_2$. It approaches $\erho = \erho_1$ when $t \to +\infty$ spiraling 
counter clockwise
around $\erho = \erho_1$ and approaches $\erho = \erho_2$ when $t \to -\infty$ spiraling clockwise around $\erho = \erho_2$.

\section{Characteristic ergospheres}

Consider a domain $\Omega$ between two ergospheres $\Gamma_1$ and $\Gamma_2$ that are also event horizons. Note that $\Gamma_1$ and $\Gamma_2$ are characteristic curves for the spatial part of $\Box_g$, but there do not exist null geodesics which travel around them. It was shown
 in [4] that $\plus$ and $\minus$ families approach $\Gamma_1$ and $\Gamma_2$ when $t \to +\infty$ and $t \to -\infty$. Consider, for example, 
 the $\plus$ family. Suppose for definiteness that any $\plus$ null-geodesic approaches $\Gamma_1$ as $t \to -\infty$ and approaches $\Gamma_2$ 
 as $t \to +\infty$. For short, we will say that $\gamma_+$ ``starts'' on $\Gamma_1$ and $\gamma_+$ ``ends'' on $\Gamma_2$. The situation
  here is similar to that of \S 6. 

Suppose now that every null geodesic of the $\plus$ family approaches $\Gamma_1$ when $t\to +\infty$ and also approaches $\Gamma_2$ when 
$t \to +\infty$, i.e. ``ends'' on $\Gamma_1$ and ``ends'' on $\Gamma_2$. Then the null geodesics of the $\minus$ family ``start'' on $\Gamma_1$ and
 $\Gamma_2$. Then as in \S 6 (cf. Theorem 6.1) using the Poincar\' e-Bendixson theorem we obtain that there exists at least one horizon belonging 
 to the $\plus$ family and one horizon belonging to the $\minus$ family. 

In the case when null-geodesics of one family ``start'' on one ergosphere and ``end'' on the other, there may be no event horizon (cf. Theorem 6.2). 

As  an  example  of two characteristic  ergospheres  consider  the Hamiltonian (1.6),  i.e.  the reduction  of  $3 + 1$  dimensional  Kerr metric  when
$\xi_\varphi=0$  and $\varphi$  is  a constant.  It was  shown  in [3]  that  $r=r_+$   and  $r=r_-$  are two ergospheres  where 
$r_\pm = m\pm\sqrt{m^2-a^2},\ 0<a<m,$  and  $r$  is defined  as  in  (1.4).   Since  $r=r_+$  and  $r=r_-$   are event horizons  in  $3+1$  Kerr metric  they are  
also  event horizons in $2+1$  reduction.  Since  there are no event  horizons between  $r=r_-$  and   $r=r_+$   the null-geodesics  of one   family  travels
from  $r_-$   to  $r_+$  when  $t$  changed  from  $-\infty$  to  $+\infty$  and  for another  family  it travels  from  $r_+$  to $r_-$  when
$-\infty<t<+\infty$. 
\begin{remark} In this paper we consider the case when the ergosphere is either not characteristic at any point or totally characteristic. In this case, event horizon curves belong to either the $\plus$ or $\minus$ family. Therefore the total set of event horizons is a sum of the contributions of the $\plus$ and $\minus$ families. In the general case some points of each ergosphere can be characteristic (cf. [4]). Then as in [4] the event horizon consists of curves belonging to either the $\plus$ or $\minus$ family. Not that as in [4] it may happen that an event horizon has corners as the result of intersecting arcs from different families.
\end{remark}

\newcommand{\etalchar}[1]{$^{#1}$}

\end{document}